\documentclass{amsart}
\usepackage[english]{babel}
\usepackage[cp1250]{inputenc}
\usepackage{amssymb,amsthm,amsfonts}
\usepackage{verbatim}
\usepackage{url}

\usepackage[dvips]{graphicx}
\usepackage{amsmath,amssymb,amsfonts, amscd}
\usepackage{pictexwd,dcpic}
\usepackage{amsthm}
\usepackage{color}
\usepackage[T1]{fontenc}
\usepackage{verbatim}
\usepackage{epstopdf}
\usepackage{enumerate}
\usepackage{url}
\usepackage[toc,page]{appendix}

\usepackage{xcolor}

\theoremstyle{theorem}
\newtheorem{theorem}{Theorem}[section]
\newtheorem{lemma}[theorem]{Lemma}
\newtheorem{corollary}[theorem]{Corollary}
\newtheorem{proposition}[theorem]{Proposition}
\theoremstyle{definition}
\newtheorem{definition}{Definition}
\newtheorem{remark}{Remark}
\newtheorem{example}{Example}

\newcommand{\PDef}{\operatorname{PDef}}
\renewcommand{\AA}{\mathbb A}

\newcommand{\RR}{\mathbb R}

\newcommand{\ZZ}{\mathbb Z}

\newcommand{\PP}{\mathbb P}

\newcommand{\spec}{\mathrm{Spec}}

\newcommand{\lan}{\langle}
\newcommand{\ran}{\rangle}

\renewcommand{\int}{\operatorname{int}}
\newcommand{\lam}{\lambda}

\newcommand{\Def}{\operatorname{Def}}

\renewcommand{\sp}{\operatorname{Span}}

\newcommand{\cA}{\mathcal{A}}

\newcommand{\cG}{\mathcal{G}}

\newcommand{\cS}{\mathcal{S}}

\newcommand{\fg}{\mathfrak{g}}
\newcommand{\fh}{\mathfrak{h}}
\newcommand{\MC}{\operatorname{MC}}
\newcommand{\G}{\operatorname{G}}

\newcommand{\HH}{\operatorname{HH}}

\newcommand{\sgn}{\operatorname{sgn}}
\renewcommand{\ker}{\operatorname{ker}}
\newcommand{\coker}{\operatorname{coker}}

\newcommand{\kss}{\scriptscriptstyle}
\newcommand{\kbb}{{\kss \bullet}}

\begin{document}

\title[Dgla controlling the Poisson deformations]{A Differential graded Lie algebra controlling the Poisson deformations of an affine Poisson variety}
\author{Matej Filip}
\address{JGU Mainz, Institut fur Mathematik, Staudingerweg 9, 55099 Mainz, Germany}
\email{mfilip@uni-mainz.de}

\begin{abstract}
We construct a differential graded Lie algebra $\fg$ controlling the Poisson deformations of an  affine Poisson variety. 
We analyse $\fg$ in the case of affine Gorenstein toric Poisson varieties. Moreover, explicit description of the second and third Hochschild cohomology groups is given for three-dimensional affine Gorenstein toric varieties.
\end{abstract}


\maketitle

\section{Introduction}
In the last decades differential graded Lie algebras have become a very important tool in deformation theory. A deformation problem is controlled by a differential graded Lie algebra $\fh$ if its corresponding functor of Artin rings is isomorphic to the deformation functor of $\fh$. In characteristic $0$ every deformation functor is controlled by a differential graded Lie algebra, due to Quillen, Deligne, Drinfeld and Kontsevich.
It is well known that  associative
non-commutative (resp. commutative) deformations of affine varieties are controlled by the Hochschild (resp. Harrison) differential graded Lie algebra.

In recent years there has been a lot of interest in Poisson deformations, i.e. in deformations of a pair consisting of a variety and a Poisson structure on it (see \cite{fre}, \cite{kalgin}, \cite{nam1}, \cite{nam2}, \cite{nam3}). 

In this paper we construct a differential graded Lie algebra $\fg$ controlling the Poisson deformations of an affine Poisson variety $\spec(A)$.  
We see that the Poisson cohomology groups $H^k(\fg)$ are related to some parts of the Hodge decomposition of Hochschild cohomology groups $\HH^n(A)$ (see e.g. \cite{ger-sch} for definition of the Hodge decomposition). In the case of affine toric varieties we gave a convex geometric description of these parts in \cite{fil}.

The paper is organized as follows. In Section 2 we recall basic deformation theory via differential graded Lie algebras and basic results about the Hodge decomposition of Hochschild cohomology groups.
The first main result of this paper is a construction of a differential graded Lie algebra $\fg$ controlling the Poisson deformations, which is done in Section 3. We notice that for the computation of the Poisson cohomology groups $H^k(\fg)$ parts of the Hodge decomposition of the Hochschild cohomology are relevant (see \eqref{eq: sp sequence 1}). 
The Poisson cohomology groups for affine Poisson Gorenstein toric surfaces are computed  in Subsection \ref{subsec pois 42}. 
Using results in \cite{fil} we explicitly compute some parts of the Hochschild cohomology groups in the case of three-dimensional affine Gorenstein toric varieties, which is our second main result (see Theorem \ref{cor prop gor}) obtained in Subsection \ref{subsection 43}. This result also reproves and generalizes \cite[Theorem 4.4]{alt}. In particular, a complete description of the second (which describes the first order deformations) and third Hochschild cohomology group (which contains the obstructions for extending deformations to larger base spaces) is given (see Corollary \ref{last zad cor}).

\section{Preliminaries}
\subsection{Deformation theory via differential graded Lie algebras}

Let $k$ be a field of characteristic $0$ and 
let $\cA$ be the category of local Artinian $k$-algebras with residue field $k$ (with local homomorphisms as morphisms). By $\cS$ we denote the category of sets. If not otherwise specified a tensor product $\otimes$ means $\otimes_k$. Let $\fg$ be a differential graded Lie algebra (dgla for short). 
By $\fg^i$ we denote the degree $i$ elements of $\fg$.

\begin{definition}
For a dgla $\fg$  we define the functor
$\text{MC}_{\fg}:\cA\to \cS$ by
$$B\mapsto \big\{ x\in \fg^1\otimes m_B~|~ d(x)+\frac{1}{2}[x,x]=0  \big\}.$$
\end{definition}
$\text{MC}_{\fg}$ is said to be the \emph{Maurer-Cartan functor} associated to $\fg$. Elements in $\MC_{\fg}(B)$ are the Maurer-Cartan elements of the dgla $\fg\otimes B$. 
\begin{definition}
Let $\cG$ denote the category of groups. Let $\fg$ be a dgla and define the functor 
$\G_{\fg}:\cA\to \cG$ given by
$$B\mapsto \text{exp}(\fg^0\otimes m_B),$$
\end{definition}
where $\text{exp}$ is the standard exponential functor on Lie algebras. $\G_{\fg}$ is said to be the \emph{gauge functor} associated to $\fg$.

Fix a dgla $\fg$ over $k$; the gauge functor $G_{\fg}$ acts naturally on the Maurer-Cartan functor $\text{MC}_{\fg}$ by the formula
$$\text{G}_{\fg}(B)\times \text{MC}_{\fg}(B)\to \text{MC}_{\fg}(B)$$
$$(e^b,x)\mapsto x+\sum_{n=0}^{\infty}\frac{[b,\cdot]^n}{(n+1)!}([b,x]-d(b)).$$
This action is called the \emph{gauge action}.

\begin{definition}
Let $\fg$ be a dgla over $k$. The \emph{deformation functor} associated to $\fg$ is the functor 
$\text{Def}_{\fg}: \cA\to \cS$ given by
$$B\mapsto \frac{\text{MC}_{\fg}(B)}{G_{\fg}(B)}.$$
\end{definition}
We say that a dgla $\fg$ \emph{controls} a functor $F$ if $\Def_{\fg}\cong F$ holds.

\subsection{The Hodge decomposition of the Hochschild cohomology}

Let $A$ be a finitely generated $k$-algebra. 
Let $C^{\kbb}(A)$ be the Hochschild cochain complex, i.e., $C^n(A)$ is the space of $k$-linear maps $f:A^{\otimes n}\to A$ (or $A$-module homomorphisms $A\otimes A^{\otimes n}\to A$) with the differential given by 
$$
\begin{array}{ll}
(df)(a_1\otimes \cdots \otimes a_n):=&a_1f(a_2\otimes \cdots \otimes a_n)+\\
&\sum^{n-1}_{i=1}(-1)^{i}f(a_1\otimes \cdots \otimes a_ia_{i+1}\otimes \cdots \otimes a_n)+\\
&(-1)^nf(a_1\otimes \cdots \otimes a_{n-1})a_n.
\end{array}
$$
The $n$-th cohomology group of this complex is called the $n$-th \emph{Hochschild cohomology group}, denoted by $\HH^n(A)$.
The Lie bracket on $C^{\kbb}(A)[1]$ is coming from the \emph{Gerstenhaber bracket} $[f,g]$ of $f\in C^m(A)$, $g\in C^n(A)$, which is defined by 
$$[f,g]:=f\circ g-(-1)^{(m+1)(n+1)}g\circ f\in C^{m+n-1}(A),$$
where 
$$(f\circ g)(a_1\otimes \cdots \otimes a_{m+n-1}):=$$
$$\sum_{i=1}^m(-1)^{(i-1)(n+1)}f(a_1\otimes \cdots \otimes a_{i-1}\otimes g(a_i\otimes \cdots \otimes a_{i+n-1})\otimes a_{i+n}\otimes \cdots \otimes a_{m+n-1}).$$
The Gerstenhaber bracket equips $C^{\kbb}(A)[1]$ with the structure of a dgla.

Gerstenhaber and Schack described the Hodge decomposition of the Hochschild (co-)homology that we will briefly recall (see \cite{ger-sch} for more details). 
In the group ring of the permutation group $S_n$ one defines $s_{i,n-i}$ to be $\sum (\sgn \pi)\pi$, where the sum is taken over those permutations $\pi\in S_n$ such that $\pi (1) <\pi (2)<\cdots<\pi (i)$ and $\pi(i+1)<\pi(i+2)<\cdots <\pi(n)$. 
Let $s_n=\sum_{i=1}^{n-1}s_{i,n-i}$.
It holds that $C^n(A)=C^n_{(1)}(A)\oplus \cdots \oplus C^n_{(n)}(A)$, where 
$C^n_{(i)}(A)=\{f\in C^n(A)~|~f\circ s_n=(2^{i}-2)f\}$.
The \emph{Hodge decomposition} is
$$\HH^n(A)\cong H^n_{(1)}(A)\oplus \cdots \oplus H^n_{(n)}(A),$$
where $H^n_{(i)}(A)$  is the $n$-th cohomology of $C^{\kbb}_{(i)}(A)$.

We denote the projectors of $\HH^n(A)$ to $H^n_{(i)}(A)$ by $e_n(i)$.
\begin{lemma}\label{lem pal lem}
For an element $p\in H^2_{(2)}(A)$ and an element $q\in H^2_{(1)}(A)$ we have the following: 
\begin{itemize}
\item the equation $e_3(3)[p,p]=0$ is the Jacobi identity, $e_3(2)[p,p]=0$
\item $[p,q]=e_3(2)[p,q]$ and $[q,q]=e_3(1)[q,q]$.
\end{itemize}
\end{lemma}
\begin{proof}
An easy computation, see also \cite{pal2}.
\end{proof}

\section{Poisson deformations}
Poisson deformations are deformations of a pair consisting of a variety and a Poisson structure on it.
Lately there has been a lot of interest in these deformations, see for example results of Namikawa \cite{nam1},\cite{nam2}, \cite{nam3} or Kaledin and Ginzburg \cite{kalgin}. 

\begin{definition}
A skew-symmetric Hochschild $2$-cocycle $p$ (i.e. $p\in C^2_{(2)}(A)$ with $dp=0$) that satisfies the Jacobi identity
$$p(a\otimes p(b\otimes c))+p(b\otimes p(c\otimes a))+p(c\otimes p(a\otimes b))=0$$
is called an (algebraic) \emph{Poisson structure} (or a \emph{Poisson bracket}). A commutative algebra
together with a Poisson bracket is called a \emph{Poisson algebra}. Its spectrum is called an \emph{affine Poisson variety}.
\end{definition}

Note that $p\in H^2_{(2)}(A)\cong \hom_A(\Omega^2_{A|k},A)$ (see e.g. \cite{lod}), where $\Omega^2_{A|k}$ is the $2$-th exterior power of the module of Kähler differentials. Using Lemma \ref{lem pal lem} we can equivalently define the Poisson structure as an element $p\in C^2_{(2)}(A)$ with $dp=e_3(3)[p,p]=0$.

\begin{definition}\label{def iso fun pois}
A \emph{Poisson deformation} of a Poisson algebra $A$ over an Artin ring $B$ is a pair $(A',\pi)$, where $A'$ is a Poisson $B$-algebra and $\pi:A'\otimes_Bk\to A$ is an isomorphism of Poisson $k$-algebras. Two such deformations $(A',\pi_1)$ and $(A'',\pi_2)$ are \emph{equivalent} if there exists an isomorphism of Poisson $B$-algebras $\phi:A'\to A''$ such that it is compatible with $\pi_1$ and $\pi_2$, i.e. such that $\pi_1=\pi_2\circ (\phi\otimes_B k)$.
\end{definition}

\subsection{Dgla that controls the deformation problem}

A functor that encodes this deformation problem is
$$\PDef_A:\cA\to \cS$$
$$B\mapsto \{\text{Poisson deformations of }A \text{ over }B\}/\sim.$$

In the following we define a dgla that controls the above deformation problem.

Consider the Double complex 1.
\begin{figure}[!htb]
$$
\begindc{\commdiag}[550]
\obj(0,0)[00]{$C^1_{(1)}(A)$}
\obj(0,1)[01]{$C^2_{(1)}(A)$}
\obj(0,2)[02]{$C^3_{(1)}(A)$}
\obj(0,3)[03]{$\vdots$}
\obj(1,0)[10]{$C^2_{(2)}(A)$}
\obj(1,1)[11]{$C^3_{(2)}(A)$}
\obj(1,2)[12]{$C^4_{(2)}(A)$}
\obj(1,3)[13]{$\vdots$}
\obj(2,0)[20]{$C^3_{(3)}(A)$}
\obj(2,1)[21]{$C^4_{(3)}(A)$}
\obj(2,2)[22]{$C^5_{(3)}(A)$}
\obj(2,3)[23]{$\vdots$}
\obj(3,0)[30]{$\cdots$}
\obj(3,1)[31]{$\cdots$}
\obj(3,2)[32]{$\cdots$}
\mor{00}{01}{$d$}
\mor{01}{02}{$d$}
\mor{02}{03}{$d$}
\mor{00}{10}{$d_p$}
\mor{10}{20}{$d_p$}
\mor{01}{11}{$d_p$}
\mor{02}{12}{$d_p$}
\mor{11}{21}{$d_p$}
\mor{20}{21}{$d$}
\mor{21}{22}{$d$}
\mor{10}{11}{$d$}
\mor{11}{12}{$d$}
\mor{12}{13}{$d$}
\mor{22}{23}{$d$}
\mor{12}{22}{$d_p$}
\mor{20}{30}{$d_p$}
\mor{21}{31}{$d_p$}
\mor{22}{32}{$d_p$}
\enddc
$$
\caption{Double complex 1}
\end{figure}

The map $d_p$ is defined as $d_p:=-[\mu_p,\cdot]:C^n(A)\to C^{n+1}(A)$, where $\mu_p\in C^2_{(2)}(A)$ is a Poisson structure on $A$. In the double complex 1 we restrict $d_p$ on the chosen domains and codomains. Note that we have $d[\mu_p,f]=[\mu_p,df]$ (since $d\mu_p=0$) and thus we really obtain a double complex.    We denote its total complex  by $D^{\kbb}$.

We define the bracket $[~,~]_p$ on $D^{\kbb}$ as follows: let $C^n(A)=C^n_{(1)}(A)\oplus \cdots \oplus C^n_{(n)}(A)$ and define
$$[\cdot,\cdot]_p:C^m(A)\times C^n(A)\to C^{m+n-1}(A)$$
$$[(f_1,...,f_m),(g_1,...,g_n)]_p:=([f_1,g_1],...,\sum_{i+j=k}[f_i,g_j],...,[f_m,g_n]),$$
where we restrict $[f_i,g_j]$ to $C^{m+n-1}_{(i+j-1)}(A)$.

This bracket defines a dgla structure on $D^{\kbb}[1]$: 
the shifted differential $d_p[1]$ is equal to $[\mu_p,\cdot]_p$
and the shifted differential $d[1]$ is equal to $[\mu,\cdot]_p$, where $\mu$ is the commutative multiplication on $A$. We denote the shifted differential of $D^{\kbb}[1]$ by $\tilde{d}$. It is given by $\tilde{d}=[\mu+\mu_p,\cdot]_p$. We can immediately check that the bracket $[~,~]_p$ and differential $\tilde{d}$ equip $D^{\kbb}[1]$ with the structure of a dgla. We denote this dgla by $C^{\kbb}_p(A)[1]$.

\begin{remark}
Note that the Gerstenhaber bracket is in general not graded with respect to the Hodge decomposition and thus the above product is not the Gerstenhaber bracket. From Lemma \ref{lem pal lem} we have $[\mu,\mu]_p=[\mu,\mu]$, $[\mu,\mu_p]_p=[\mu,\mu_p]$ and $[\mu_p,\mu_p]_p=e_3(3)[\mu_p,\mu_p]$.
\end{remark}

After applying the differentials $d$ on the double complex 1, we obtain for $j,k\geq 1$ the first spectral sequence
\begin{equation}\label{eq: sp sequence 1}
E_1^{j,k}=H^{j+k-1}_{(j)}(A)\Rightarrow H^{j+k-1}(C^{\kbb}_p(A)[1]),
\end{equation}
where $d_1=-[\mu_p,\cdot]: E_1^{j,k}\to E_1^{j+1,k}$.

To show that the functor $\PDef_A$ is controlled by the dgla $C^{\kbb}_p(A)[1]$ we first need few Lemmata. For a $k$-algebra $A$ we define the $k$-algebra $A_0$ that is as a $k$-vector space isomorphic to $A$ and it has zero multiplication. 
\begin{lemma}\label{lem pois mc2}
Poisson algebra structures on $A_0$
are in bijection with Maurer-Cartan elements of $C^{\kbb}_p(A_0)[1]$, i.e with elements $(\mu,\mu_p)\in C_{(1)}^2(A_0)\oplus C_{(2)}^2(A_0)$ satisfying $\frac{1}{2}[\mu,\mu]=[\mu,\mu_p]=\frac{1}{2}[\mu_p,\mu_p]_p=0$.
\end{lemma}
\begin{proof}
Let $(\mu,\mu_p)$ be a Maurer-Cartan element of $C^{\kbb}_p(A_0)[1]$. 
We define the multiplication on $A_0$ by $a\cdot b:=\mu(a,b):=\mu(a\otimes b)$ and the Poisson structure by $\{a,b\}:=\mu_p(a,b):=\mu_p(a\otimes b)$. The product $\cdot$ is commutative and associative if and only if  $\mu\in C_{(1)}^2(A_0)$ and $\frac{1}{2}[\mu,\mu]=0$. Now we show that 
$\mu_p$ defines a Poisson structure. Since $\mu_p\in C_{(2)}^2(A_0)$, everything except the Jacobi identity is clear. The Jacobi identity we get from $\frac{1}{2}[\mu_p,\mu_p]_p=0$ as in Lemma \ref{lem pal lem} (note that we have $[\mu_p,\mu_p]_p=e_3[\mu_p,\mu_p]$). 
We now show the following claim: 
$$\{a,b\cdot c\}=\{a,b\}c+\{a,c\}b~(\text{i.e.\ }\mu_p(a,\mu(b,c))=\mu(\mu_p(a,b),c)+\mu(\mu_p(a,c),b))$$ holds if and only if $[\mu,\mu_p]=0$. Assume that
$$F(a,b,c):=\mu_p(a,\mu(b,c))-\mu(\mu_p(a,b),c)-\mu(\mu_p(a,c),b)=0$$ 
holds.
We have
$$
\begin{array}{l}
F(a,b,c)+F(c,a,b)=\\
\big(\mu_p(a,\mu(b,c))-\mu(\mu_p(a,b),c)-\mu(\mu_p(a,c),b)\big)+\big(\mu_p(c,\mu(a,b))-\mu(\mu_p(c,a),b)-\mu(\mu_p(c,b),a)\big)=\\
-[\mu_p,\mu].
\end{array}
$$
and thus we see one direction. For the other direction we compute
$$
\begin{array}{l}
[\mu_p,\mu](a,b,c)+[\mu_p,\mu](a,c,b)-[\mu_p,\mu](b,a,c)=\\
\big( \mu_p(ab,c)-\mu_p(a,bc)+\mu_p(a,b)c- \mu_p(b,c)a \big)+\big( \mu_p(ac,b)-\mu_p(a,cb)+\mu_p(a,c)b- \mu_p(c,b)a \big)-\\
\big( \mu_p(ba,c)-\mu_p(b,ac)+\mu_p(b,a)c- \mu_p(a,c)b \big)=\\
2\big( -\mu_p(a,bc)+\mu_p(a,b)c+\mu_p(a,c)b  \big)=-2F(a,b,c).
\end{array}
$$
To shorten the notation we wrote $ab=\mu(a,b)$ and similarly for other elements. 
Thus the claim is proved. From this we easily conclude the proof.
\end{proof}

\begin{definition}
The \emph{Poisson product} on a vector space $V$ is a pair $(\cdot,\{~,~\})$, such that $(V,\cdot,\{~,~\})$ is a Poisson algebra.
\end{definition}

\begin{lemma}\label{lem bij mc pois}
Let $A$ be a Poisson algebra and let $B$ be an Artin ring. Maurer-Cartan elements of $C_p^{\kbb}(A\otimes m_B)[1]$ are in bijection with Poisson products on the vector space $A\otimes_k B$, giving the known Poisson product on $A\cong A\otimes_k B/m_B$.
\end{lemma}
\begin{proof}
 Let a Maurer-Cartan element $(\mu,\mu_p)$ of $C^{\kbb}_p(A_0)[1]$ represents the Poisson bracket  of $A$.
The Poisson products on the vector space $A\otimes_k B$, giving the known product on $A\cong A\otimes_k B/m_B$ are obtained by $(\xi,\xi_p)\in C^2_{(1)}(A\otimes m_B)\oplus C^2_{(2)}(A\otimes m_B)$ satisfying 
\begin{equation}\label{eq mc mc pops}
[(\mu,\mu_p)+(\xi,\xi_p),(\mu,\mu_p)+(\xi,\xi_p)]_p=0.
\end{equation}
 Since $[(\mu,\mu_p),(\mu,\mu_p)]_p=0$ and the differential on $C_p^{\kbb}(A\otimes m_B)[1]$ is given by $[(\mu,\mu_p),\cdot]$, then we see that the equation \eqref{eq mc mc pops} gives us MC elements $(\xi,\xi_p)$ of $C_p^{\kbb}(A\otimes m_B)[1]$.
\end{proof}

\begin{proposition}\label{prop pois dgla}
For a Poisson algebra $A$ the functor $\PDef_A$ is controlled by the dgla $C^{\kbb}_p(A)[1]$.
\end{proposition}
\begin{proof}
We write for short $\fg:=C^{\kbb}_p(A)[1]$.
By Lemma \ref{lem bij mc pois} there exists a bijection between $\MC_{\fg}(B)$ and Poisson products on the vector space $A\otimes_k B$, giving the known Poisson product on $A\cong A\otimes_k B/m_B$.

To conclude the proof we show that two Poisson products $(\cdot,\{~,~\})$ and $(\cdot',\{~,~\}')$ on $A\otimes_k B$ are equivalent (in the sense of Definition \ref{def iso fun pois}) if and only if the corresponding elements $(\gamma,\gamma_p), (\gamma',\gamma_p')\in \MC_{\fg}(B)$ are gauge equivalent. The products are equivalent if and only if there exists $\alpha\in C^1(A)\otimes m_B$ such that 
\begin{equation}\label{eq com mult}
a\cdot'b=\exp(\alpha)(\exp(-\alpha)(a)\cdot \exp(-\alpha)(b)),
\end{equation}
\begin{equation}\label{eq pois mult}
\{a,b\}'=\exp(\alpha)(\{\exp(-\alpha)(a), \exp(-\alpha)(b)\}).
\end{equation}
As above let a Maurer-Cartan element $(\mu,\mu_p)$ of $C^{\kbb}_p(A_0)[1]$ represents the Poisson bracket of $A$.

From \eqref{eq com mult} we obtain
\begin{equation}\label{eq poisson mcg1}
(\mu+\gamma')(a,b)=\exp(\alpha)(\exp(-\alpha)(a)\cdot \exp(-\alpha)(b))=\exp([\alpha,\cdot])(\mu+\gamma)(a,b),
\end{equation}
where the later equality we get after some elementary computation.
In the same way from \eqref{eq pois mult} we obtain
\begin{equation}\label{eq poisson mcg2}
(\mu_p+\gamma_p')(a,b)=\exp(\alpha)\{\exp(-\alpha)(a), \exp(-\alpha)(b)\}=\exp([\alpha,\cdot])(\mu_p+\gamma_p)(a,b).
\end{equation}

Elements $(\gamma,\gamma_p)\in \MC_{\fg}(B)$ and $(\gamma',\gamma_p')\in \MC_{\fg}(B)$ are gauge equivalent if
\begin{equation}\label{eq poisson mcg}
(\gamma',\gamma_p')=(\gamma,\gamma_p)+\sum_{n=0}^{\infty}\frac{[\alpha,\cdot]_p^n}{(n+1)!}([\alpha,(\gamma,\gamma_p)]_p-\tilde{d}(\alpha))
\end{equation}
holds.

Since $\tilde{d}(\alpha)=[(\mu,\mu_p),\alpha]_p=-[\alpha,(\mu,\mu_p)]_p$ and $[\alpha,\cdot]=[\alpha,\cdot]_p$, we see that \eqref{eq poisson mcg} holds if and only if the equations \eqref{eq poisson mcg1} and \eqref{eq poisson mcg2} hold.
\end{proof}

\section{Computation of the Hochschild and Poisson cohomology groups for Gorenstein toric varieties}

\subsection{Affine Gorenstein toric varieties}
Let $M,N$ be mutually dual, finitely generated, free Abelian groups. We denote by $M_{\RR}$, $N_{\RR}$ the associated real vector spaces obtained via base change with $\RR$. Let $\sigma=\lan a_1,...,a_N \ran\subset N_{\RR}$ be a rational, polyhedral cone with apex in $0$ and let $a_1,...,a_N\in N$ denote its primitive fundamental generators (i.e. none of the $a_i$ is a proper multiple of an element of $N$). We define the dual cone $\sigma^{\vee}:=\{r\in M_{\RR}~|~\lan \sigma,r \ran\geq 0\}\subset M_{\RR}$ and denote by $\Lambda:=\sigma^{\vee}\cap M$ the resulting semi-group of lattice points. Its spectrum $\text{Spec}(k[\Lambda])$ is called an \emph{affine toric variety}.

\emph{Affine toric Gorenstein varieties}
are obtained by putting a lattice polytope $P\subset \AA\cong \RR^{n-1}$ into the affine hyperplane $\AA\times \{1\}\subset \AA\times \RR=:N_{\RR}$ and defining $\sigma:=\text{Cone}(P)$, the cone over $P$. Then the canonical degree $R^*\in M$ equals $(\underline{0},1)$.

It is a trivial check that Hochschild differentials respect the grading given by the degrees $R\in M$. Thus we get the Hochschild subcomplex $C^{\kbb,R}_{(i)}$ and we denote the corresponding cohomology groups by $H^{n,R}_{(i)}(A)\cong T^{n-i,R}_{(i)}(A)$, where the later is the degree $R$ part of the (higher) Andr\'e-Quillen cohomology group $T^{n-i,R}_{(i)}(A)$ (see \cite[Section 4]{fil}). We will not use general Andr\'e-Quillen cohomology theory, we will only use the well-known isomorphism $T^{n-i}_{(i)}(A)\cong H^n_{(i)}(A)$ for $n\geq i$ (see e.g. \cite{lod}).


\subsection{Poisson cohomology groups of Poisson Gorenstein toric surfaces}\label{subsec pois 42}

Let $X_{\sigma_n}=\spec(A_n)$ be the Gorenstein toric surface given by $g(x,y,z)=xy-z^{n+1}$. $\Lambda_n:=\sigma_n^{\vee}\cap M$ is generated by $S_1:=(0,1)$, $S_2:=(1,1)$ and $S_3:=(n+1,n)$, with the relation $S_1+S_3=(n+1)S_2$. 

In order to compute the Poisson cohomology groups we need to analyse the spectral sequence \eqref{eq: sp sequence 1}. First we need to understand all the parts of the Hochschild cohomology. 

\begin{proposition}\label{prop eq pr}
It holds that
\begin{equation}\label{eq dimt1}
\dim_kT^{1,-R}_{(1)}(A_n)=\dim_kT^{1,-R}_{(2)}(A_n)=\left\{
\begin{array}{ll}
1& \text{ if }R=kS_2 \text{ for }2\leq k\leq n+1\\ 
0 & \text{ otherwise. }
\end{array}
\right.
\end{equation}
Moreover,
 $T^2_{(1)}(A_n)\cong H^3_{(1)}(A_n)=0$.  
For $i\geq 3$ we have $T^k_{(i)}(A_n)=0$ if $k\ne i-1,i$ and $$T^{i-1}_{(i)}(A_n)\cong T^{i}_{(i)}(A_n)\cong A_n/(\frac{\partial g}{\partial x_1},\frac{\partial g}{\partial x_2},\frac{\partial g}{\partial x_3}).$$
The later has $k$-dimension equal to $n$. 
\end{proposition}
\begin{proof}
\cite[Proposition 3.3, Example 3]{fil}.
\end{proof}
\begin{corollary}
Since $T^{i-1}_{(i)}(A_n)\cong H^{2i-1}_{(i)}(A_n)$ and $T^{i}_{(i)}(A_n)\cong H^{2i}_{(i)}(A_n)$ we see that $E^{j,k}_2=E^{j,k}_{\infty}$ holds for every $j,k\geq 1$.
\end{corollary}

Elements from $H^2_{(2)}(A_n)$ define Poisson structures on $\spec(A_n)$ by Lemma \ref{lem pal lem}, since $H^3_{(3)}(A)=0$.
Let $\mu_p\in H^2_{(2)}(A_n)$ denote a Poisson structure on $\spec(A_n)$.
 Let $\fg_n:=C^{\kbb}_p(A_n)[1]$. 
From above we have the following description of the  spectral sequence \eqref{eq: sp sequence 1}:
$$E_1^{3,\kbb}: 0\xrightarrow{d_1}H^4_{(2)}(A_n)\xrightarrow{d_1}H^5_{3}(A_n)\xrightarrow{d_1}
0\xrightarrow{d_1} \cdots$$
 $$E_1^{2,\kbb}: H^2_{(1)}(A_n)\xrightarrow{d_1}H^3_{(2)}(A_n)\xrightarrow{d_1}0\xrightarrow{d_1}
0\xrightarrow{d_1} \cdots$$
 $$E_1^{1,\kbb}: H^1_{(1)}(A_n)\xrightarrow{d_1}H^2_{(2)}(A_n)\xrightarrow{d_1}0\xrightarrow{d_1}
0\xrightarrow{d_1} \cdots,$$

 $E_1^{j,\kbb}$ for $j>3$ have only two non-vanishing terms $E_1^{j,j-1}=H^{2j-2}_{(j)}(A_n)$ and  $E_1^{j,j}=H^{2j-1}_{(j)}(A_n)$.

\begin{corollary}\label{cor tor sur}
$$H^0(\fg_n)\cong\ker\big(H^1_{(1)}(A_n)\xrightarrow{d_1}H^2_{(2)}(A_n)\big),$$
$$H^1(\fg_n)\cong\coker\big(H^1_{(1)}(A_n)\xrightarrow{d_1}H^2_{(2)}(A_n)\big)\oplus \ker\big(H^2_{(1)}(A_n)\xrightarrow{d_1}H^3_{(2)}(A_n)\big),$$
for $k\geq 2$ we have 
$$H^k(\fg_n)\cong\left\{
\begin{array}{ll}
\coker\big(H^{k}_{(\frac{k}{2})}(A_n)\xrightarrow{d_1}H^{k+1}_{(\frac{k}{2}+1)}(A_n)\big)& \text{ if }k\text{ is even}\\ 
\ker\big(H^{k+1}_{(\frac{k+1}{2})}(A_n)\xrightarrow{d_1}H^{k+2}_{(\frac{k+1}{2}+1)}(A_n)\big) & \text{ if }k\text{ is odd}.
\end{array}
\right.
$$
\end{corollary}

\begin{proposition}\label{prop gor tor sur}
It holds that $H^2(\fg_n)\cong A_n/(\frac{\partial g}{\partial x_1},\frac{\partial g}{\partial x_2},\frac{\partial g}{\partial x_3})$.
\end{proposition}
\begin{proof}
From Corollary \ref{cor tor sur} we know that $H^2(\fg_n)\cong\coker\big(H^{2}_{(1)}(A_n)\xrightarrow{d_1}H^{3}_{(2)}(A_n)\big)$. In \cite{fil2} we proved that the Gerstenhaber product $H^2_{(1)}(A_n)\times H^2_{(2)}(A_n)\to H^3_{(2)}(A_n)$ is the zero map. Since by definition $d_1=-[\mu_p,\cdot]$, we see that $d_1$ is the zero map. Thus $H^2(\fg_n)\cong H^{3}_{(2)}(A_n)\cong A_n/(\frac{\partial g}{\partial x_1},\frac{\partial g}{\partial x_2},\frac{\partial g}{\partial x_3})$.
\end{proof}

\begin{example}
For every hypersurface given by a polynomial $g(x,y,z)$ in $k^3$, we can define a Poisson structure $\pi_g$ on the quotient $k[x,y,z]/g$, namely:
$$\pi_g:=\partial_x(g)\partial_y\wedge \partial_z+\partial_y(g)\partial_z\wedge\partial_x+\partial_z(g)\partial_x\wedge \partial_y,$$
i.e., we contract the differential 1-form $dg$ to $\partial_x\wedge \partial_y\wedge \partial_z$.
In the case of Gorenstein toric surfaces $X_{\sigma_n}=\spec(A_n)$ we have that $$\pi_g=f_0(\lam_1,\lam_2)x^{-S_2+\lam_1+\lam_2},$$
where $f_0$ is skew-symmetric and bi-additive with $f_0(S_1,S_3)=-(n+1)$ (see \cite[Example 4]{fil}). Thus we see that $\pi_g\in H^{2,-S_2}_{(2)}(A_n)$. In this case we see that $H^1(\fg_n)\cong H^2(\fg_n)\cong A_n/(\frac{\partial g}{\partial x_1},\frac{\partial g}{\partial x_2},\frac{\partial g}{\partial x_3})$ by the proof of Proposition \ref{prop gor tor sur} and Corollary \ref{cor tor sur} since $H^1_{(1)}(A_n)\xrightarrow{d_1}H^2_{(2)}(A_n)$ is surjective. This is special case of \cite[Lemma 3.1]{kalgin}.
\end{example}

\subsection{The Hochschild cohomology of three dimensional affine Gorenstein toric varieties}\label{subsection 43}
For an affine Gorenstein toric variety $X_\sigma=\spec(A)$ we will explicitly compute $T^1_{(i)}(A)$ for all $i\geq 1$. They appear in $E_1^{\kbb,2}$ (see \eqref{eq: sp sequence 1} and note that $H^{j+1}_{(j)}(A)\cong T^1_{(j)}(A)$) and they are also important ingredients for understanding $\HH^2(A)$ and $\HH^3(A)$ (see Corollary \ref{last zad cor}). This subsection reproves and generalizes \cite[Theorem 4.1]{alt}.

In \cite{fil} we obtained a convex geometric description of $T^1_{(i)}(A)$ for $i\geq 1$, which we recall now.
Let the cone $\sigma=\lan a_1,...,a_N\ran$ represent an $n$-dimensional toric variety $X_{\sigma}=\spec(A)$, $n\geq 3$. 
For $R\in M$ we define the affine space $$\AA(R):=\{a\in N_{\RR}~|~\langle a,R \rangle=1\}\subset N_{\RR}$$ and consider the polyhedron 
$Q(R):=\sigma\cap \AA(R)\subset \AA(R).$
Vertices of $Q(R)$ are $\bar{a}_j:=a_j/\lan a_j,R  \ran$, for all $j$ satisfying $\langle a_j,R \rangle\geq 1$. We denote $T^1_{(i)}(-R):=T^{1,-R}_{(i)}(A)$.

Let $d_{jk}:=\overline{\bar{a}_j\bar{a}_k}$ denote the compact edges of $Q(R)$ (for $\lan a_j,a_k\ran\leq \sigma$, $\lan a_j,R\ran\geq 1$, $\lan a_k,R\ran\geq 1$). 
We denote the lattice $N\cap\sp_k\lan a_j,a_k\ran$ by $\bar{N}_{jk}$ and its dual with $\bar{M}_{jk}$. Let  $\bar{R}_{jk}$ denote the projection of $R$ to $\bar{M}_{jk}$. 
By $T^1_{\lan a_j,a_{k}\ran}(-\bar{R}_{jk})$ we denote the degree $-\bar{R}_{jk}$ part of the toric surface given by a cone $\lan a_j,a_{k}\ran$.
We define $\sp_kK^R_{jk}$ to be 
$$\sp_kK^R_{jk}:=\sp_k\big(K^R_{a_j}\cap K^R_{a_k}\big),$$
with $K^R_{a_j}=\{r\in \Lambda~|~\lan a_j,r \ran<\lan a_j,R \ran\}$.
Let
$$W_j(R):=\left \{
\begin{array}{lr}
2& \text{ if }\lan a_j,R \ran>1\\ 
1& \text{ if }\lan a_j,R \ran=1\\
0& \text{ if }\lan a_j,R \ran\leq 0.
\end{array}
\right.
$$

\begin{proposition}\label{t1i}
If the compact part of $Q(R)$ lies in a two-dimensional affine space we have
$$\dim_kT_{(i)}^{1}(-R)=\max\big\{0,\sum_{j=1}^NV^i_j(R)-\sum_{d_{jk}\in Q(R)}Q^i_{jk}(R)-{n\choose i}+s^i_{Q(R)}\big\},$$
where

$$V^i_j(R):=\left \{
\begin{array}{ll}
{n\choose i}& \text{ if }\lan a_j,R \ran>1\\ 
{n-1\choose i}& \text{ if }\lan a_j,R \ran=1\\
0& \text{ if }\lan a_j,R \ran\leq 0,
\end{array}
\right.
$$
$$Q^i_{jk}(R):=\left \{
\begin{array}{ll}
{W_j(R)+W_k(R)+n-4-\dim_kT^{1}_{\lan a_j,a_{k}\ran}(-\bar{R}_{jk})\choose i} & \text{ if }\lan a_j,R \ran,\lan a_k,R \ran\ne 0 \\
0&\text{ otherwise,}
\end{array}
\right.
$$
$$s^i_{Q(R)}:=
\left \{
\begin{array}{ll}
\dim_k\wedge^i\big(\bigcap_{d_{jk}\in Q(R)}\sp_kK^R_{jk}\big)& \text{ if }Q(R) \text{ is compact}\\
0&\text{ otherwise.}
\end{array}
\right.
$$
\end{proposition}
\begin{proof}
See \cite[Proposition 4.14]{fil}.
\end{proof}

From now on we assume that $X_\sigma$ is a three-dimensional affine toric Gorenstein variety given by a cone $\sigma=\lan a_1,...,a_N \ran$, where $a_1,...,a_N$ are arranged in a cycle. Let $d_j:=a_{j+1}-a_j$ ($a_{N+1}:=a_1$)  and let $\ell(j):=\ell(d_j)$ denote its lattice lenght. 
Let $s_1,...,s_N$ be the fundamental generators of the dual cone $\sigma^{\vee}$, labelled so that $\sigma \cap (s_j)^{\perp}$ equals the face spanned by $a_j,a_{j+1}\in \sigma$. Using the previous notation we see that the polytope $P=Q(R^*)$ is a polygon with (oriented) edges equal to $d_j$ for $j=1,...,N$. It holds that $\lan a_j,R^*\ran=1$ for all $j=1,...,N$.

\begin{example}\label{ex prvi}
A typical example of a non-isolated, three dimensional toric Gorenstein singularity is the affine cone $X_{\sigma}$ over the weighted projective space $\PP(1,2,3)$. The cone $\sigma$ is given by 
$\sigma=\lan a_1,a_2,a_3\ran$, where 
$$a_1=(-1,-1,1),~~~ a_2=(2,-1,1),~~~a_3=(-1,1,1).$$
We obtain $\sigma^{\vee}=\lan s_1,s_2,s_3 \ran$ with 
$$s_1=(0,1,1),~~~s_2=(-2,-3,1),~~~s_3=(1,0,1).$$
\end{example}

We need to better understand $s^i_{Q(R)}$ and $\dim_kT^1_{\lan a_j,a_{j+1}\ran}(-\bar{R}_{jk})$ that appears in $Q^i_{j,j+1}(R)$. 

\begin{lemma}\label{lem zadazad}
Let a cone $\sigma=\lan a_j,a_{j+1} \ran\subset \bar{N}_{j,j+1}$ define a toric surface given by the edge $d_j$.
We have 
$$\dim_k\sp_kK^R_{j,j+1}=\max\{0,W_j(R)+W_{j+1}(R)-2-\dim_kT_{\lan a_j,a_j\ran}^{1}(-\bar{R}_{j,j+1})\}.$$
\end{lemma}
\begin{proof}
See \cite[Lemma 4.3]{fil}.
\end{proof}

\begin{lemma}\label{lem t1udz}
$$\dim_kT_{\lan a_j,a_{j+1}\ran}^{1}(-\bar{R}_{j,j+1})=\left\{
\begin{array}{cc}
1&\text{ if }2\leq \lan a_j,R\ran=\lan a_{j+1},R\ran\leq \ell(j)\\
0&\text{otherwise}.
\end{array}
\right.$$
\end{lemma}
\begin{proof}
The toric surface $\lan a_j,a_{j+1}\ran$ is isomorphic to the Gorenstein toric  surface $\spec(A_{\ell(j)-1})$ and then the proof follows from Proposition \ref{prop eq pr} (more precisely from equation \eqref{eq dimt1}).
\end{proof}

\begin{lemma}\label{lem pom tr}
If there exists $a_j$ such that $\lan R,a_j\ran\leq 0$, then $s^i_{Q(R)}=0$, otherwise $s^i_{Q(R)}\leq {3\choose i}$.
\end{lemma}
\begin{proof}
Follows trivially from definitions.
\end{proof}

The next lemma establishes a useful criterion when $T^1_{(i)}(-R)$ is zero in the Gorenstein three-dimensional case.
\begin{lemma}\label{lem lepaz}
Assume that $\lan a_j,R\ran\ne \lan a_{j+1},R \ran$ for all $j=1,...,N$ ($a_{N+1}:=a_1$). Then $T^1_{(i)}(-R)=0$.
\end{lemma}
\begin{proof}
We will use Proposition \ref{t1i}.  By the assumption and Lemma \ref{lem t1udz} we know that $\dim_kT^1_{\lan a_j,a_{j+1}\ran}(-\bar{R}_{j,j+1})=0$ for all $j$. Since $\lan a_j,R\ran=1$ for at most  two $j\in \{1,...,N\}$ we can using Lemma \ref{lem pom tr} easily see  that 
$$\sum_{j=1}^NV^i_j(R)-\sum_{j=1}^N{W_j(R)+W_{j+1}(R)-1\choose i}-{3\choose i}+s^i_{Q(R)}\leq 0,$$
which implies that $T^1_{(i)}(-R)=0$ for all $i$. 
\end{proof}

\begin{lemma}\label{lem silem}
Let $R=qR^*$ for $q\geq 2$. It holds that $\dim_k\cap_j\sp_kE^R_{j,j+1}=3$ if $\ell(j)<q$  for all $j$. Moreover, $\dim_k\cap_j\sp_kE^R_{j,j+1}=2$ if $\ell(j)<q$  for all $j$ except two (denoted by $j_1$ and $j_2$), for which it holds that $d_{j_1}$ and $d_{j_2}$ are parallel (the case $j_1=j_2$ is included).  Otherwise it holds that
$
\dim_k\cap_j\sp_kK^R_{j,j+1}=1.
$
\end{lemma}
\begin{proof}
By definition we easily see that
\begin{equation}\label{eq in lem}
\sp_kK^R_{j,j+1}=
\left\{
\begin{array}{cc}
M\otimes_\ZZ k&\text{ if }\ell(j)< q\\
\sp_k\{a_j^\perp\cap a_{j+1}^\perp,R^*\}&\text{ if }\ell(j)\geq  q. 
\end{array}
\right.
\end{equation}
Since $\sp_k\{a_j^\perp\cap a_{j+1}^\perp,R^*\}=\{c\in M\otimes_\ZZ k~|~\lan c,a_j \ran=\lan c,a_{j+1} \ran\}$ we see that $\sp_k\{a_{j_1}^\perp\cap a_{{j_1}+1}^\perp,R^*\}=\sp_k\{a_{j_2}^\perp\cap a_{{j_2}+1}^\perp,R^*\}$ for $j_1,j_2\in \{1,..,N\}$ if and only if $j_1=j_2$ or $d_{j_1}$ is parallel to $d_{j_2}$. Thus we can easily conclude the proof.
\end{proof}

Let $\int(\sigma^\vee)$ denotes the interior of $\sigma^\vee$.

\begin{theorem}\label{cor prop gor}
Let $X_\sigma$ be a three-dimensional affine toric Gorenstein variety. 
The following holds:
\begin{enumerate}
\item 
 $T^1_{(1)}(-R)$ is non-trivial in the following cases:
\begin{enumerate}[(a)]
\item $R=R^*$ with $\dim_kT^1_{(1)}(-R)=N-3$,

\item $R=qR^*$ (for $q\geq 2$) with $\dim_kT^1_{(1)}(-R)=\max\{0,\#\{j~|~q\leq \ell(j)\}-2\}$,

\item $R=qR^*-ps_j$ with $2\leq q \leq \ell(j)$ and $p\in \ZZ$ sufficiently large such that $R\not \in \text{int}(\sigma^{\vee})$. In this case $\dim_kT^1_{(1)}(-R)=1$.

\end{enumerate}

Additional degrees exist only in the following two (overlapping) exceptional cases:

\begin{enumerate}[(d)]

\item $P$ contains a pair of parallel edges $d_j$, $d_k$, both longer than every other edge. Then $\dim_kT^1_{(1)}(-qR^*)=1$ for $q$ in the range 
$$\max\{\ell(l)~|~l\ne j,k\}<q\leq \min\{\ell(j),\ell(k)\}\},$$

\item[(e)] $P$ contains a pair of parallel edges $d_j$, $d_k$ (with distance $d=\lan a_j,s_k \ran=\lan a_k,s_j \ran$). In this case $\dim_kT^1_{(1)}(-R)=1$ for $R=qR^*+ps_j$ with $1\leq q\leq \ell(j)$ and $1\leq p\leq (\ell(k)-q)/d$.
\end{enumerate}

\item
$T^1_{(2)}(-R)$ is non-trivial in the following cases:
\begin{enumerate}[(a)]
\item $R=R^*$ with $\dim_kT^1_{(2)}(-R)=N-3,$

\item $R=qR^*$ (for $q\geq 2$) with $\dim_kT^1_{(2)}(-R)=\max\{0,2\cdot \#\{j~|~q\leq \ell(j)\}-3\}$,

\item $R=qR^*-ps_j$ with $2\leq q \leq \ell(j)$ and $p\in \ZZ$ sufficiently large such that $R\not \in \text{int}(\sigma^{\vee})$. In this case $\dim_kT^1_{(2)}(-R)=2$.

\end{enumerate}
Additional degrees exist only in the following two (overlapping) exceptional cases:

\begin{enumerate}[(d)]

\item $P$ contains a pair of parallel edges $d_j$, $d_k$, both longer than every other edge. Then $\dim_kT^1_{(2)}(-qR^*)=2$ for $q$ in the range 
$$\max\{\ell(l)~|~l\ne j,k\}<q\leq \min\{\ell(j),\ell(k)\}\},$$

\item[(e)] 
$P$ contains a pair of parallel edges $d_j$, $d_k$ (with distance $d=\lan a_j,s_k\ran=\lan a_k,s_j \ran$). In this case $\dim_kT^1_{(2)}(-R)=2$ for $R=qR^*+ps_j$ with $1\leq q\leq \ell(j)$ and $1\leq p\leq (\ell(k)-q)/d$. 
\end{enumerate}

\item
$T^1_{(3)}(-R)$ is non-trivial in the following cases:

\begin{enumerate}[(b)]
\item $R=qR^*$ (for $q\geq 2$) with $\dim_kT^1_{(3)}(-R)=\max\{0,\#\{j~|~q\leq \ell(j)\}-1\}$,

\item[(c)] $R=qR^*-ps_j$ with $2\leq q \leq \ell(j)$ and $p\in \ZZ$ sufficiently large such that $R\not \in \text{int}(\sigma^{\vee})$. In this case $\dim_kT^1_{(3)}(-R)=1$.
\end{enumerate}
Additional degrees exist only in the following two (overlapping) exceptional cases:

\begin{enumerate}[(d)]

\item $P$ contains a pair of parallel edges $d_j$, $d_k$, both longer than every other edge. Then $\dim_kT^1_{(3)}(-qR^*)=1$ for $q$ in the range 
$$\max\{\ell(l)~|~l\ne j,k\}<q\leq \min\{\ell(j),\ell(k)\}\},$$

\item[(e)] 
$P$ contains a pair of parallel edges $d_j$, $d_k$ (with distance $d=\lan a_j,s_k \ran=\lan a_k,s_j \ran$). In this case $\dim_kT^1_{(3)}(-R)=1$ for $R=qR^*+ps_j$ with $1\leq q\leq \ell(j)$ and $1\leq p\leq (\ell(k)-q)/d$.
\end{enumerate}

\item
We have $T^1_{(i)}(-R)=0$ for $i\geq 4$.
\end{enumerate}
\end{theorem}
\begin{proof}
We distinguish the following cases.
\begin{itemize}
\item Let $R=R^*$. 
In this case we have $\lan a_j,R\ran$=1 and $\sp_kK^R_{a_j}=(a_j)^\perp$ for all $j$. Thus 
 $s^i_{Q(R)}=0$ for all $i$ and by Lemma \ref{lem t1udz} we have $T^1_{\lan a_j,a_{j+1}\ran}(-\bar{R}_{j,j+1})=0$ for all $j$. Moreover, 
 $$\sum_{j=1}^NV^i_j(R)-\sum_{d_{jk}\in Q(R)}Q^i_{jk}(R)=N{2\choose i}-N{1\choose i}.$$
  From Proposition \ref{t1i} it follows that $\dim_kT^1_{(1)}(-R^*)=\dim_kT^1_{(2)}(-R^*)=N-3$ and $T^1_{(i)}(-R^*)=0$ for $i>2$. Thus we proved (a) cases (note that $T^1_{(3)}(-R^*)=0$ and thus (a) case does not appear in the case (3)).

\item Let $R=qR^*$, where $q\geq 2$.

In this case we have $\lan a_j,R\ran\geq 2$.
Thus $\sum_{j=1}^NV^i_j(R)={3\choose i}N$. 
Let us define $v:=\#\{j~|~q\leq \ell(j)\}$.

For $i=1$ 
we have $\sum_{j=1}^NQ^1_{j,j+1}(R)=3N-v$ (since for $q\leq \ell(j)$ we have $\dim_kT^1_{\lan a_j,a_{j+1}\ran}(-\bar{R}_{j,j+1})=1$ by Lemma \ref{lem t1udz}). 
Thus  $\dim_kT^1_{(1)}(-R)=v-3+s^1_{Q(R)}$ holds by Proposition \ref{t1i}.  If $s^1_{Q(R)}=1$ (see Lemma \ref{lem silem} when this holds) we obtain the case (1b). There is an exceptional case for which $\dim_kT^1_{(1)}(-R)\ne 0$: this appears if $v=2$ and $s^1_{Q(R)}=2$. From Lemma \ref{lem silem} we see that this happens for $q$ in the range 
$$\max\{\ell(l)~|~l\ne j,k\}<q\leq \min\{\ell(j),\ell(k)\}\}.$$
Thus we proved the case (1d).

In the case $i=2$ we have $\sum_{j=1}^NQ^2_{j,j+1}(R)={2\choose 2}v+{3\choose 2}(N-v)=3N-2v$. Thus Proposition \ref{t1i} gives us that $\dim_kT_{(2)}^1(-R)=2v-3+s^2_{Q(R)}$. As in the case $i=1$ we obtain from Lemma \ref{lem silem} the cases (2b) and (2d).

For $i=3$ we have $\sum_{j=1}^NQ^3_{j,j+1}(R)=N-v$. By Proposition \ref{t1i} we have $\dim_kT_{(3)}^1(-R)=v-1+s^3_{Q(R)}$. As in the case $i=1$ we obtain from Lemma \ref{lem silem} the cases (3b) and (3d).

\item Let $R\not \in \int(\sigma^\vee)$. By Lemma \ref{lem lepaz} we see that the only possible cases for having a non-zero $T^1_{(i)}(-R)$ occur when $\lan a_j,R\ran=\lan a_{j+1},R\ran> 0$ for some $j\in \{1,...,N\}$ and $\lan a_l,R\ran\leq 0$ for all other $l$. This happens for
$R=qR^*-ps^j$ with $q\geq 1$ and $p\in \ZZ$ sufficiently large such that $R\not \in \text{int}(\sigma^{\vee})$. In this case we have $\lan a_j,R\ran=\lan a_{j+1},R\ran=q$ and $\lan a_l,R\ran\leq 0$ for other $l$. 
If $q=1$, then by Lemma \ref{lem t1udz} it holds that $T^1_{\lan a_j,a_{j+1} \ran}(-\bar{R}_{j,j+1})=0$ and thus by Proposition \ref{t1i} we have 
$$\dim_kT^1_{(i)}(-R)=\max\{0,2{2\choose i}-{1\choose i}-{3\choose i}\}=0$$
for all $i$. If $q\geq 2$, then using Lemma \ref{lem t1udz} we see that
$$\dim_kT^1_{(i)}(-R)=
\left\{
\begin{array}{cc}
2{3\choose i}-{2\choose i}-{3\choose i}&\text{ if }2\leq q\leq \ell(j) \\
2{3\choose i}-{3\choose i}-{3\choose i}=0&\text{ if }q>\ell(j).
\end{array}
\right.$$
In the cases $2\leq q\leq \ell(j)$ we see that $\dim_kT^1_{(1)}(-R)=\dim_kT^1_{(3)}(-R)=1$, $\dim_kT^1_{(2)}(-R)=2$ and $\dim_kT^1_{(i)}(-R)=0$ for $i\geq 4$. This proves (c) cases.

\item Let $R\in \int(\sigma^\vee)$ and $R\ne qR^*$ for some $q\geq 1$. 
By Lemma \ref{lem lepaz} it follows that $T^1_{(i)}(-R)=0$ for all $R$, except maybe for $R=qR^*+ps_j$ for some $j$ since in this case we have $\lan a_j,R\ran=\lan a_{j+1},R\ran=q$. 

Let us first assume that $q\geq 2$. In this case we have that
$$\dim_kT^1_{(i)}(-R)=N{3\choose i}-\sum_{l=1}^N{3-\dim_kT^1_{\lan a_l,a_{l+1}\ran}(-\bar{R}_{l,l+1})\choose i}-{3\choose i}+s^i_{Q(R)}.$$
We see that $T^1_{(i)}(-R)=0$ if $T^1_{\lan a_l,a_{l+1}\ran}(-\bar{R}_{l,l+1})=0$ for all $l\in \{1,...,N\}$ since $s^i_{Q(R)}\leq {3\choose i}$. If only $T^1_{\lan a_j,a_{j+1}\ran}(-\bar{R}_{j,j+1})\ne 0$ we still obtain $T^1_{(i)}(-R)=0$ since in this case $s^i_{Q(R)}\leq {2\choose i}$ (by Lemma \ref{lem zadazad}). 
Thus we see that the only case to obtain nontrivial $T^1_{(i)}(-R)$ is when there exist parallel edges $d_j$, $d_k$ (with distance $d=\lan a_j,s_k \ran=\lan a_k,s_j \ran$). We have $\dim_kT^1_{(i)}(-R)\ne 0$ for $R=qR^*+ps_j$ with $2\leq q\leq \ell(j)$ and $1\leq p\leq (\ell(k)-q)/d$ since in this case $\dim_kT^1_{\lan a_j,a_{j+1}\ran}(-\bar{R}_{j,j+1})=\dim_kT^1_{\lan a_k,a_{k+1}\ran}(-\bar{R}_{j,j+1})=1$ by Lemma \ref{lem t1udz}.
Thus we have
$$\dim_kT^1_{(i)}(-R)=N{3\choose i}-\big(2{2\choose i}+(N-2){3\choose i}\big)-{3\choose i}+{2\choose i}={3\choose i}-{2\choose i}.$$
We see that in this case $\dim_kT^1_{(1)}(-R)=\dim_kT^1_{(3)}(-R)=1$ and $\dim_kT^1_{(2)}(-R)=2$. Similarly we can treat the case $q=1$ and thus finish the proof.
\end{itemize}
\end{proof}

\begin{remark}
Note that in the case $i=1$ our formulas agree with the ones given in \cite[Theorem 4.1]{alt}, which were obtained by different methods. 
\end{remark}

\begin{example}
Let $X_\sigma$ be as in Example \ref{ex prvi}. From Theorem \ref{cor prop gor} we obtain that
if $R\in \{2R^*-\alpha s_3, 2R^*-\beta s_1, 2R^*-\gamma s_1~|~\alpha\geq 1, \beta\geq 1, \gamma \geq 2\}$, then $\dim_kT^1_{(1)}(-R)=\dim_kT^1_{(3)}(-R)=1$ and $\dim_kT^1_{(2)}(-R)=2$. For other degrees $S\in M$ we have $T^1_{(i)}(-S)=0$ for all $i\geq 1$.
\end{example}

\begin{corollary}\label{last zad cor}
Let $\spec(A)$ be a three-dimensional affine toric Gorenstein variety. 
The Hodge decomposition gives us
$$\HH^2(A)\cong T^1_{(1)}(A)\oplus T^0_{(2)}(A),$$
$$\HH^3(A)\cong T^2_{(1)}(A)\oplus T^0_{(3)}(A)\oplus T^1_{(2)}(A).$$
Descriptions of $T^0_{(2)}(A)$ and $T^0_{(3)}(A)$ were given in \cite{fil}. The module $T^2_{(1)}(A)$ was analysed in \cite[Corollary 5.4]{klaus}. Theorem \ref{cor prop gor} gives us an explicit description of $T^1_{(1)}(A)$ and $T^1_{(2)}(A)$ and thus we complete understanding of the second (which describes the first order associative non-commutative deformations) and third Hochschild cohomology group (which contains the obstructions for extending associative non-commutative deformations to larger base spaces).
\end{corollary}

\section*{Acknowledgements}
I would like to thank to my PhD advisor Klaus Altmann for his constant support and many useful conversations.


\begin{thebibliography}{44}

\bibitem{alt}
{K. Altmann:  \emph{One parameter families containing three-dimensional toric Gorenstein singularities}, Explicit birational geometry of 3-folds, London Math. Soc. Lecture Note Ser., vol. \textbf{281}, Cambridge Univ. Press, Cambridge (2000), 21--50.}

\bibitem{klaus}
{K. Altmann, A. B. Sletsj\o e: \emph{Andr\'e-Quillen cohomology of monoid algebras}, J.  Alg. \textbf{210} (1998), 1899--1911.}

\bibitem{fil}
{M. Filip: \emph{Hochschild cohomology and deformation quantization of affine toric varieties}, J. Alg. \textbf{508} (2018), 188--214}.


\bibitem{fil2}
{M. Filip: \emph{The Gerstenhaber product $\HH^2(A)\times \HH^2(A)\to \HH^3(A)$ of affine toric varieties}, 	arXiv:1803.07486}



\bibitem{fre}
{B. Fresse: \emph{Homologie de Quillen pour les algebres de Poisson}, C.R. Acas. Sci. Paris \textbf{326} (1998), 1053--1058}.


\bibitem{ger-sch}
{M. Gerstenhaber, S.D. Schack: \emph{A Hodge-type decomposition for commutative algebras}, J. Pure Appl. Alg. \textbf{48} (1987), 229--247.}




\bibitem{kalgin}
{V. Ginzburg, D. Kaledin: \emph{Poisson deformations of symplectic quotient singularities}, Adv. in Math. \textbf{186}, (2004), pp. 1-57}.

\bibitem{lod}
{J.-L. Loday: \emph{Cyclic homology}, Grundlehren der mathematischen Wissenschaften \textbf{301}, Springer-Verlag, (1992).}

\bibitem{nam1}
{Y. Namikawa: \emph{Flops and Poisson deformations of symplectic varieties}, Publ. RIMS, \textbf{44} (2008), 259-314}.


\bibitem{nam2}
{Y. Namikawa: \emph{Poisson deformations of affine symplectic varieties} Duke Math. J. 156 (2011) 51-85}.

\bibitem{nam3}
{Y. Namikawa: \emph{Poisson deformations and Birational Geometry}, J. Math. Sci. Univ. Tokyo \textbf{22} (2015), 339--359}.

\bibitem{pal2}
{V.P. Palamodov: \emph{Infinitesimal deformation quantization of complex analytic spaces}, Lett. Math. Phys. \textbf{79} (2007), iss. 2, 131--142.}




\end{thebibliography}
\end{document}